\documentclass{amsart}
\usepackage[utf8]{inputenc}

\usepackage{amsfonts, amstext, amsmath, amsthm, amscd, amssymb}
\usepackage[dvipsnames,usenames]{xcolor}
\usepackage{aligned-overset}
\usepackage{soul} 

\usepackage{fullpage}
\usepackage{tikz}
\usepackage{float}
\usepackage{wrapfig}
\usepackage{caption}
\usepackage{pinlabel}
\usepackage[labelfont=bf]{caption}
\DeclareCaptionFormat{myformat}{\fontsize{9}{10}\selectfont#1#2#3}
\usepackage{hyperref}
\usepackage[nameinlink]{cleveref} 
\hypersetup{colorlinks=true,linkcolor=NavyBlue,citecolor=Plum}

\newtheorem{thm}{Theorem}[section]

\newtheorem{cor}[thm]{Corollary}
\newtheorem{lemma}[thm]{Lemma}
\newtheorem{prop}[thm]{Proposition}

\newtheorem{defn}[thm]{Definition}
\newtheorem{rmk}[thm]{Remark}
\newtheorem*{rmk*}{Remark}

\AtBeginDocument{\jot=6pt}

\title[$n$-bridge braids and the braid index]{$n$-bridge braids and the braid index}

\author[Gollero]{Dane Gollero}
\address{Department of Mathematics, University of Utah}
\email{3rundane@gmail.com}

\author[Krishna]{Siddhi Krishna}
\address{Department of Mathematics, Columbia University}
\email{siddhi.krishna@bc.edu}

\author[Loving]{Marissa Loving}
\address{Department of Mathematics, University of Wisconsin - Madison}
\email{mloving2@wisc.edu}

\author[Neri]{\\Viridiana Neri}
\address{Department of Mathematics, Columbia University}
\email{vjn2108@columbia.edu}

\author[Tahir]{Izah Tahir}
\address{School of Mathematics, Georgia Institute of Technology}
\email{itahir3@gatech.edu}

\author[White]{Len White}
\address{Department of Mathematics, Portland State University}
\email{len\textunderscore white@live.com}

\definecolor{darkpurple}{rgb}{.5,0,.5}
\definecolor{darkteal}{rgb}{0,.5,.5}
\definecolor{darkgreen}{rgb}{.15,.5,.35}

\begin{document}

\maketitle

\begin{abstract}
    In this work, we find a closed form formula for the braid index of an $n$-bridge braid, a class of positive braid knots which simultaneously generalizes torus knots, 1-bridge braids, and twisted torus knots. Our proof is elementary, effective, and self-contained, and partially recovers work of Birman--Kofman. Along the way, we show that the disparate definitions of twisted torus knots in the literature agree. 
\end{abstract}

\section{Introduction} \label{section:Introduction}
\subsection{Motivation and Summary}

Knots and links play an important role in low-dimensional topology. One simple way to measure the complexity of a link $L$ in $S^3$ is the \textit{braid index}, $i(L)$, which is the minimum number of strands required to represent $L$ as the closure of a braid on as many strands. As every link is realized as the closure of some braid \cite{Alexander}, the braid index is a well defined link invariant. Even for knots, the braid index is often quite difficult to compute. The simplest infinite family for which the braid index is computed are the $T(p,q)$ torus knots for which $i(T(p,q))=\min\{p,q\}$. Analogous formulas in the literature are rare.

It is natural to hope that generalizations of torus knots lend themselves to closed braid index formulas. 
One axis along which we can generalize comes from the Dehn surgery perspective. Dehn surgery is a powerful operation within 3--manifold topology: every 3--manifold is obtained by Dehn surgery along some link in $S^3$ \cite{Lickorish:DehnSurgery, Wallace:DehnSurgery}. Despite the ubiquity of this technique, some of the most basic questions about Dehn surgery remain open. For example, the infamous \textit{Berge Conjecture} predicts exactly which knots in $S^3$ admit a Dehn surgery to \textit{lens spaces}, the rational homology 3--spheres admitting genus--1 Heegaard splittings \cite{Berge}. Moser \cite{Moser:TorusKnots} showed that torus knots always admit Dehn surgeries to lens spaces. Lens spaces are examples of \textit{L-spaces}: the closed, connected, oriented 3-manifolds with ``small'' Heegaard Floer homology \cite{OSz:LensSpaceSurgeries}. It immediately follows that torus knots are examples of knots admitting a Dehn surgery to L-spaces. Thus, one way to generalize torus knots would be to identify other knots which \textit{also} admit Dehn surgeries to L-spaces. 

Perhaps surprisingly, there are infinitely many hyperbolic knots which admits surgeries to lens spaces: the first examples were identified by Fintushel and Stern a decade after Moser's work \cite{FintushelStern}. A decade later still, work of Berge and Gabai showed that an infinite sub-family of \textit{1-bridge braids} admit a Dehn surgery to a lens space \cite{Berge:SolidTorus, Gabai:SolidTori, Gabai:1BridgeBraids} (a precise definition of these knots appears later in this paper). In fact, all 1-bridge braids admit a Dehn surgery to L-spaces \cite{GLV:11LSpace}. Therefore, we see that 1-bridge braids are a generalization of torus knots from the Dehn surgery perspective -- moreover, they are a natural extension from a braid-theoretic point of view as well (see \Cref{section:background} for more details). Besides 1-bridge braids, there are other braid theoretic ways to generalize torus knots, including \textit{$n$-bridge braids} \cite{Gabai:1BridgeBraids}, \textit{twisted torus knots} \cite{BirmanKofman, Vafaee:TwistedTorusKnots, LeeVafaee}, and \textit{T-links} \cite{BirmanKofman}. \Cref{section:background} contains the definitions of these various families, and the relationships between them.

Braid theoretic definitions are valuable, in part, because they are explicit and concrete -- however, it can be remarkably difficult to determine whether different braid theoretic definitions coincide. For example, twisted torus knots have received a lot of attention over the past few years \cite{BirmanKofman, ChampanerkarFuterKofmanNeumannPurcell, Vafaee:TwistedTorusKnots, LeeVafaee, KrishnaMorton, DePaiva}, yet there are multiple different braid theoretic definitions of twisted torus knots scattered throughout the literature. In this paper, on route to proving our main result, we prove that these various definitions of twisted torus knots coincide; see \Cref{section:LorenzKnots}.

As mentioned above, $n$-bridge braids (which we define in \Cref{section:background}) are one natural generalization of twisted torus knots from a braid theoretic standpoint. Torus knots are defined by using two parameters; in contrast, $n$-bridge braids are defined using four parameters. In this work, we compute the braid index of any $n$-bridge braid.

\begin{thm} \label{thm:NBridge}
The braid index of an $n$-bridge braid, $\mathcal{K}(w,b,t,n)$, is determined by the defining parameters; namely,
\begin{center}
$ i(\mathcal{K}(w,b,t,n)) = \begin{cases} 
      w &  t \geq w, \ n\geq 1 \\
      t &  w > t > b, \ n \geq 1 \\
      t+1 & w > b \geq t, n=1\\
      b+1 & w > b\geq t, n+t \geq b+1, n>1\\
      n+t & w > b \geq t, n+t < b+1, n>1.
      \end{cases}$
\end{center}
\end{thm}

\noindent As an immediate consequence, we determine the braid index of a 1-bridge braid:

\begin{cor} \label{thm:BraidIndex}
The braid index of a 1--bridge braid $\mathcal{K}(w, b,t)$ is:
\begin{center}
$ i(\mathcal{K}(w,b,t)) = \begin{cases} 
      w & t \geq w \\
      t &  w > t > b \\
      t+1 & b \geq t.
      \end{cases}$
\end{center}
\end{cor}

The main proof strategy for \Cref{thm:NBridge}, and \Cref{thm:BraidIndex} is elementary: we use the well-known \textit{Markov moves} to manipulate the presentation of the braid, and then apply a result of Morton and Franks--Williams \cite{Morton:PolysFromBraids, Morton:KnotPolys, FranksWilliams}. Their theorem says that if a positive braid $\beta$ on $k$ strands contains a \textit{positive full twist}, then in fact, $i(\beta) = k$. Our proof is completely effective: we concretely apply Markov moves to produce an explicit positive braid which contains a full twist; we then apply the Morton---Franks--Williams result to this braid to know the braid index.

\Cref{thm:NBridge} partially recovers -- using very different techniques -- a result of Birman--Kofman \cite{BirmanKofman}. In \cite{BirmanKofman}, the authors define \textit{T-links} (these links are the closures of particular positive braids), and prove that the set of T-links coincides with the well studied \textit{Lorenz links}, i.e. the set of links which can be embedded onto the ``Lorenz template'', which is seen in \Cref{fig:Lorenz}. Lorenz links are interesting in their own right as they exhibit rich dynamical and geometric properties  \cite{BirmanWilliams, BirmanKofman, ChampanerkarFuterKofmanNeumannPurcell, Birman:Lorenz, Dehornoy:Lorenz, dePaivaPurcell}. 
Notably, Birman--Kofman show that over half of the ``simplest" hyperbolic knots are Lorenz knots \cite{BirmanKofman}. 

\begin{figure}[htb]\center
\includegraphics[scale=0.5]{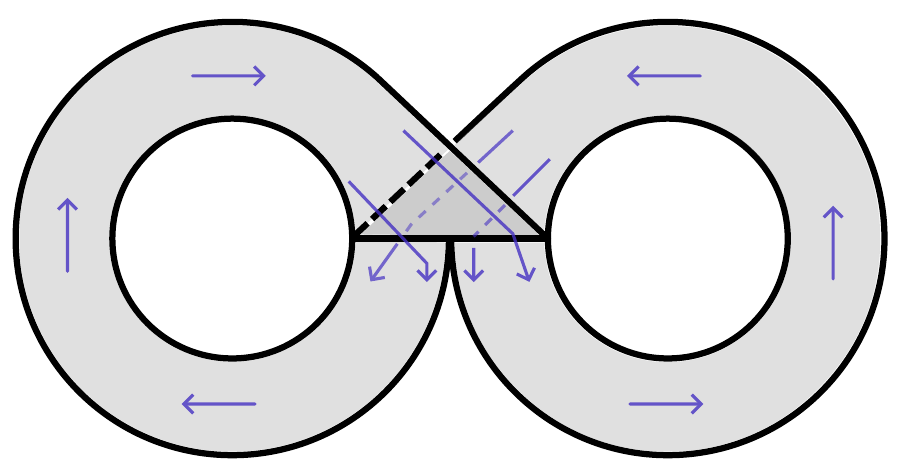}
\caption{In this figure, we see the \textit{Lorenz template}, which is a 2-complex with some extra dynamical information. The arrows on the template dictate how a simple curve (or collection thereof) should flow around the surface. For example, a curve can flow from the left to the right by passing in the ``front'' of the branch locus, and it can flow from the right to the left by passing ``behind'' the branch locus. Simple closed curves that can be embedded on the Lorenz template are called \textit{Lorenz links}.}
\label{fig:Lorenz}
\end{figure}

We coarsely summarize the Birman--Kofman strategy for computing the braid index for T-links and then contrast it with the methods used in this paper. Birman--Williams \cite{BirmanWilliams} proved that Lorenz knots can always be realized as the closures of positive braids which contain a positive full twist -- therefore, one can apply the Morton---Franks--Williams theorem to determine the braid index. So, Birman--Kofman first prove that T-links coincide with Lorenz knots, and then adapt the T-link presentation to a Lorenz presentation; applying Birman--Williams yields the final result. In contrast to their combinatorial and dynamical proof, our proof is self-contained, elementary, and explicit, as we bypass the Lorenz template and only utilize Markov moves. Moreover, unlike Birman--Kofman, our proof \textit{produces an explicit braid} which is Markov equivalent to an $n$-bridge braid. This itself has value, and was utilized by Krishna--Morton to study 4-dimensional properties of Lorenz knots \cite{KrishnaMorton}; next, we briefly describe some of their work, and the ties to this paper.

Recently, Krishna--Morton  showed that if a knot $K$ can be realized as the closure of a positive braid with a full twist, then the braid index of $K$ appears as the third exponent in the Alexander polynomial for $K$ \cite[Theorem 1.2]{KrishnaMorton}. This already yields applications for 1-bridge braids: in the proof of \Cref{thm:BraidIndex}, we show that 1-bridge braids can be realized as the closure of a positive braid with a full twist and thus, by \cite[Theorem 1.2]{KrishnaMorton}, the third exponent of the Alexander polynomial for a 1-bridge braid can be determined directly from the braid index formula in \Cref{thm:BraidIndex}. (We note that, in general, it is very hard to determine non-trivial terms in the Alexander polynomial of a positive braid knot.) Prior to our work, if one wanted to compute the braid index of a 1-bridge braid, one would have to do the following: (1) show that a 1-bridge braid is a T-link (from Gabai's definition of 1-bridge braids, and Birman--Kofman's definition of T-links, this is not clear), and then (2) apply Birman--Kofman (and Birman--Williams) to determine the braid index.  

Therefore, in addition to identifying a closed formula for the braid index, our paper accomplishes a few important goals: it unifies multiple viewpoints and definitions in the literature, and it is elementary and effective (and could be implemented by a computer for more complicated links). Perhaps most importantly, it produces an \textit{explicit} positive braid word to which the Morton---Franks--Williams theorem applies.

\subsection{Outline of the paper} In \Cref{section:background} we outline the definitions and foundational results that we will use throughout the paper and set some notational conventions for the remainder of the paper. In \Cref{section:LorenzKnots}, we prove that the different definitions of twisted torus knots in the literature agree, and also show that $n$-bridge braids (as we defined them) are Lorenz knots. In \Cref{section:preliminaries}, we establish a series of lemmas and propositions to be used in the proof of \Cref{thm:NBridge}. The proof of \Cref{thm:NBridge} is contained in \Cref{section:n-bridge}.

\subsection{Acknowledgements} This work began as part of the 2021 iteration of the Georgia Tech School of Math's \textit{Research Experience for Undergraduates} program. We gratefully acknowledge support from NSF grants DMS-1552285 (SK, ML), DMS-1745583 (DG, SK, VN, IT, LW), DMS-2103325 (SK), and DMS-1902729 (ML). We also thank Kyle Hayden for helpful conversations, and Zipei Nie for informing us that there is overlap between the preliminary lemmas in \Cref{section:preliminaries} and \cite{Nie:1BBsSatellites}. See \Cref{section:preliminaries} for further details. Finally, we sincerely thank the referee for their detailed and thoughtful comments.

\section{Background} \label{section:background}

We begin with some preliminaries. 

\begin{defn} The \textbf{braid group on $n$ strands}, denoted $B_n$, is the group with the following presentation:
\begin{align*}
B_n : = \langle \sigma_1, \sigma_2, \ldots, \sigma_{n-1} \ | \ 
\mathcal{R} \ \rangle
\end{align*}

where $\mathcal{R}$ denotes the following set of braid relations:
    \begin{enumerate}
        \item $\sigma_i \sigma_j = \sigma_j \sigma_i$ if $|i - j| > 1$
        \item $\sigma_i \sigma_{i+1} \sigma_i = \sigma_{i+1} \sigma_i \sigma_{i+1}$ where $1 \leq i \leq n-2$.
    \end{enumerate}
\end{defn}

This is also known as \textit{Artin's presentation for the braid group}, and the generating set $\sigma_1, \ldots, \sigma_{n-1}$ are typically referred to as the \textit{Artin generators} for the braid group. There are other group presentations for the braid group. The interested reader can consult \cite{BirmanBrendle} for a survey, and to discover some of the many connections between the braid group and topology, geometry, algebra, and dynamics.

\begin{rmk} \label{garside}
\textup{In \cite{Garside}, Garside proves that the center of $B_n$ is generated by the full twist; that is, the element $(\sigma_1 \sigma_2 \ldots \sigma_{n-1})^n = (\sigma_{n-1} \ldots \sigma_2 \sigma_1)^n$ commutes with every other element in $B_n$. In the same work, Garside defines the \textbf{Garside element}: for the braid group $B_n$, the Garside element $\Delta_n$ is defined as follows: $\Delta_n = (\sigma_1 \sigma_2 \dots \sigma_{n-1})(\sigma_1 \ldots \sigma_{n-2}) \ldots (\sigma_1 \sigma_2)(\sigma_1)$. He notes that $(\Delta_n)^2$ is the full twist, and that $\sigma_i \Delta_n = \Delta_n \sigma_{n-i}$. These facts about the braid group will be useful in our proofs. For more about the Garside element, we recommend \cite{GonzalesMeneses:BraidGroups} as a reference.}
\end{rmk}

\begin{defn}
A braid $\beta \in B_n$ is a \textbf{positive braid}, or \textbf{braid positive}, if it contains only positive Artin generators. A knot or link is \textbf{braid positive} if it can be realized as the closure of a positive braid.
\end{defn}

\begin{defn} 
A \textbf{T-link} is a link which is realized as the closure of a positive braid $\tau$, where 
\begin{align} \label{TLink}
\tau = (\sigma_1 \sigma_2 \ldots \sigma_{p_1-1})^{q_1} (\sigma_1 \sigma_2 \ldots \sigma_{p_2-1})^{q_2} \ldots (\sigma_1 \sigma_2 \ldots \sigma_{p_s-1})^{q_s}.
\end{align}
Here, $2 \leq p_1 \leq p_2 \leq \ldots \leq p_s$, $0 < q_i$ for all $i$, and $\tau$ is a braid in $B_{p_s}$. 
\end{defn}

\begin{defn}[\`a la Vafaee \cite{Vafaee:TwistedTorusKnots}] \label{defn:TTK}
A \textbf{twisted torus knot} is realized as the closure of a positive braid $\omega$ on $n$ strands, where 
\begin{align}
\omega = (\sigma_{n-1} \sigma_{n-2} \ldots \sigma_2 \sigma_1)^p (\sigma_{n-1} \sigma_{n-2} \ldots \sigma_{n-k+1})^{qk}.
\end{align}
Here, $3 \leq n$, $2 \leq p$, $2 \leq k \leq n-1$, and $q \geq 1$.
That is, adding $q$ many positive full twists into $k$ adjacent strands of a positive torus knot yields a twisted torus knot.
\end{defn}

We note that we do not want to consider the case where $k=n$: if $k = n$, then the definition of $\omega$ in \Cref{defn:TTK} simplifies to the standard braid word for the torus link $T(n, p + qk) = T(n, p+ qn)$.

\begin{defn} \label{def:n-bridgeBraid}
An \textbf{$n$-bridge braid}, denoted $\mathcal{K}(w,b,t,n)$, is the link realized as the closure of the positive braid \[(\sigma_b \sigma_{b-1} \ldots \sigma_{1})^n(\sigma_{w-1} \ldots \sigma_2 \sigma_1)^t.\]
Here, $3 \leq w, \ 1 \leq b \leq w-2, \ t \leq 2$, and $1 \leq n$.
Qualitatively, $w$ is the number of strands on which the braid is presented, $b$ is the bridge length, $t$ is the number of twists, and $n$ is the number of bridges.
\end{defn}

Note: we do not want $b = w-1$: if this were permitted, then the braid word in \Cref{def:n-bridgeBraid} would simplify to the torus knot $T(w, n+t)$.

The family of $1$-bridge braids (e.g. where $n=1$ in \Cref{def:n-bridgeBraid}) are especially well studied: as we noted in \Cref{section:Introduction}, 1-bridge braids have been studied by Berge, Gabai, and Greene-Lewallen-Vafaee \cite{Berge:SolidTorus, Gabai:SolidTori, Gabai:1BridgeBraids, GLV:11LSpace}, amongst others. \Cref{fig:VennDiagram} organizes how 1-bridge braids, twisted torus knots, $n$-bridge braids, and T-links are related.

\begin{figure}[htb]\center
\labellist \tiny
\pinlabel {1-bridge} at 110 190
\pinlabel {braids} at 110 176
\pinlabel {twisted} at 280 204
\pinlabel {torus} at 280 190
\pinlabel {knots} at 280 176
\pinlabel {$n$-bridge} at 410 190
\pinlabel {braids} at 410 176
\pinlabel {Lorenz} at 555 238
\pinlabel {knots} at 555 224
\pinlabel {T-links} at 555 150
\pinlabel \rotatebox{90}{$\longleftrightarrow$} at 555 188
\endlabellist
\includegraphics[scale=0.4]{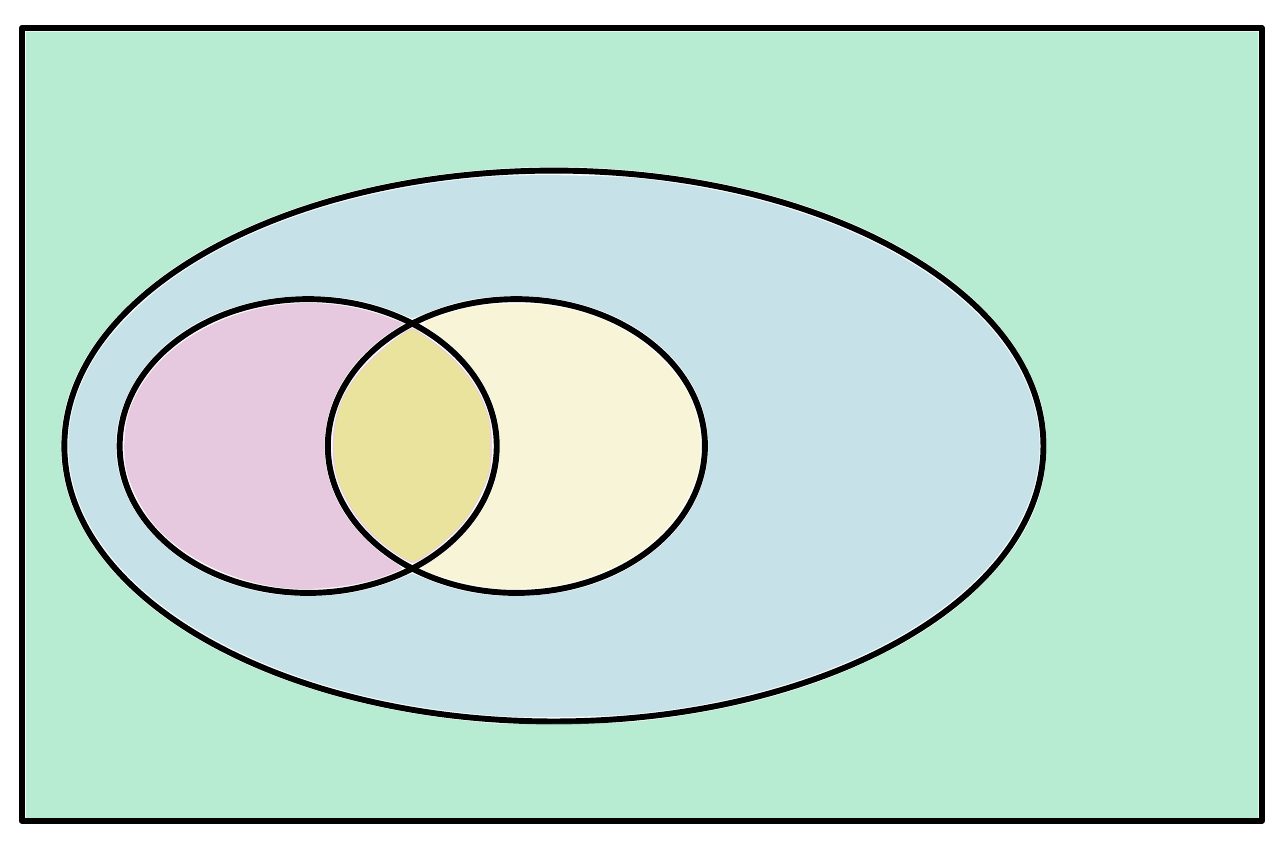}
\caption{A schematic explaining how the relevant families of knots are related. We emphasize that $n$-bridge braids can be viewed as a generalization of twisted torus knots: the hypothesis that there are $q$ full twists on $k$ adjacent strands is weakened to partial twists.}
\label{fig:VennDiagram}
\end{figure}

\begin{figure}[htb]
\begin{tikzpicture}
    \draw (1, 0) node[inner sep=0] {\includegraphics[scale=0.3]{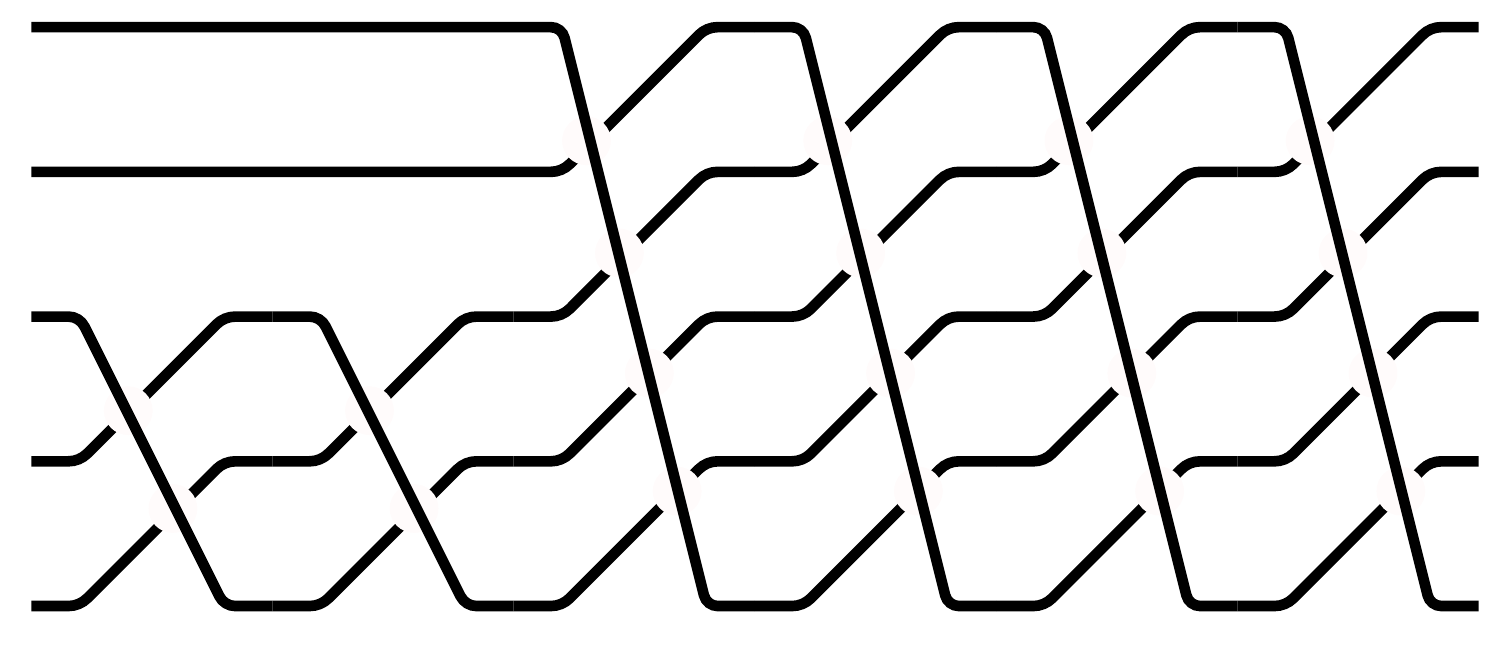}};
    \draw (-3, -1.4) node {$1$};
    \draw (-3, -.7) node {$2$};
    \draw (-3, 0) node {$3$};
    \draw (-3, .8) node {$4$};
    \draw (-3, 1.5) node {$5$};
\end{tikzpicture}
\captionof{figure}{The $n$-bridge braid $K(5, 4,2,2)$ is realized as the closure of this braid.}
\label{fig:Example}
\end{figure}

\noindent Note that we will use $\mathcal{K}(w,b,t,n)$ to denote both the link and the associated braid word \[(\sigma_b \sigma_{b-1} \ldots \sigma_{1})^n(\sigma_{w-1} \ldots \sigma_2 \sigma_1)^t.\]

To compute the braid index of a link $L$, we need a method for decreasing the number of strands in the braided presentation of $L$. This method is called \emph{destabilization}.

\begin{defn}
    Let $\omega$ be a braid word on $n$ strands. A \textbf{stabilization} replaces $\omega$ with $\omega \sigma_n$ or $\omega \sigma_n ^{-1}$, a braid word on $n+1$ strands. The reverse operation $($of replacing $\omega \sigma_n$ or $\omega \sigma_n^{-1}$, where $\omega$ has no $\sigma_n^{\pm 1}$ letters$)$ is called \textbf{destabilization}.
\end{defn}

If two braids have the same closure, then the braids must be related in a particular way. 

\begin{thm} [Markov \cite{Markov}] \label{thm:Markov}
Let $\beta_1$ and $\beta_2$ be two braid words. Then, their braid closures are isotopic if and only if $\beta_1$ and $\beta_2$ are related by any combination of: (1) braid relations, (2) conjugations, and (3) (de)stabilizations.
\end{thm}

In particular, Markov's theorem tells us the following: if $\alpha$ and $\beta$ are braids in $B_n$, then the braids $\alpha \sigma_n \beta$ and $\alpha \beta$ (which are braids in $B_{n+1}$) have isotopic closures as links in $S^3$. We will use this observation at various points throughout the proof of our main theorem.



Finally, to determine the braid index, we will use a result independently obtained by Morton and Franks--Williams. 

\begin{thm} [Morton \cite{Morton:PolysFromBraids, Morton:KnotPolys}; Franks--Williams \cite{FranksWilliams}] \label{thm:FMW} 
Suppose $\beta \in B_n$ is a positive braid, and $\beta = \omega (\sigma_{n-1} \ldots \sigma_1)^n$, where $\omega$ is a positive braid word. Then the braid index of $\beta$ is $n$, i.e. $i(\beta) = n$.
\end{thm}

As noted in \Cref{garside}, Garside proved that the positive full twist commutes with every other element in the braid group. In particular, combining with the Morton-Franks-Williams result, we see the following: if $\alpha, \beta \in B_n$, and $\alpha$ and $\beta$ are both positive braid words, then $\alpha (\sigma_{n-1} \ldots \sigma_1)^n \beta$ has braid index $n$.

\subsection{Conventions} Throughout the paper, we will indicate how the braid word changes by underlining the letters of the braid word as they are changed by braid relations, conjugations, or de-stabilizations. When we draw our braids vertically, we read them top-to-bottom. When we draw our braids horizontally, we read them from left-to-right. For us, $\sigma_i$ corresponds to strand $(i+1)$ crossing over strand $(i)$. Given a braid $\beta$, the notation $\hat \beta$ will denote its closure. Finally, we will use ``$=$" to denote that two sides of an equation are isotopic as braid closures and so are equal up to braid relations and Markov moves.

\section{$n$-bridge braids are Lorenz knots} \label{section:LorenzKnots}

Birman--Kofman \cite{BirmanKofman} showed that the class of Lorenz links coincides with that of T-links. By the Birman--Kofman conventions \cite[Equation 1]{BirmanKofman}, a twisted torus knot on $\ell$ strands is realized as the closure of the following braid:
\begin{align*}
\beta_{BK} = (\sigma_1 \sigma_2 \ldots \sigma_{r})^{kr} (\sigma_1 \sigma_2 \ldots \sigma_{\ell-1})^s.
\end{align*}

However, our definition of twisted torus knots (in \Cref{defn:TTK}) follows Vafaee's conventions \cite{Vafaee:TwistedTorusKnots}; he defines a twisted torus knot on $n$ strands to be obtained by taking the braid closure of:
\begin{align*}
\beta_V = (\sigma_{n-1} \sigma_{n-2} \ldots \sigma_2 \sigma_1)^p (\sigma_{n-1} \sigma_{n-2} \ldots \sigma_{n-k})^{qk}.
\end{align*}

It is not immediate that these braid words are Markov equivalent (and hence that their closures are isotopic knots in $S^3$). Given this discrepancy in the literature, we explicitly show that the Vafaee and Birman--Kofman twisted torus knots are Markov equivalent. The remainder of this section is devoted to this proof: we explicitly use Markov moves to put twisted torus knots and $n$-bridge braids into T-link form.

\begin{lemma} \label{lemma:LeftMove}
Fix some $w \geq 3$. Let $t \geq 2, a \geq 2,$ and $c \geq 1$. 
Let $\alpha_1 = (\sigma_a \sigma_{a+1}\ldots \sigma_{a+c})(\sigma_1 \sigma_2 \ldots \sigma_{w-1})^t$ and let $\alpha_2 =  (\sigma_{a-1} \sigma_{a}\ldots \sigma_{a+c-1})(\sigma_1 \sigma_2 \ldots \sigma_{w-1})^t$, where $\alpha_1$ and $\alpha_2$ are both elements of the braid group $B_w$. Then $\alpha_1$ and $\alpha_2$ are conjugate braids. In particular, $\widehat{\alpha_1}$ and $\widehat{\alpha_2}$ are isotopic links in $S^3$. 
\end{lemma}

\begin{proof} We do some explicit braid moves to verify the claim. For clarity, we underline the portions of the braid that are being transformed from one line to the next. We set $\gamma_w :=(\sigma_1 \sigma_2 \ldots \sigma_{w-1})$, a braid word in $B_w$. We begin by pushing some terms to the right:
\begin{align*}
\alpha_1 &= (\sigma_{a}\sigma_{a+1}\ldots \sigma_{a+c-1} \sigma_{a+c})(\sigma_1 \sigma_2 \ldots \sigma_{w-1})^t \\
&= (\sigma_{a}\sigma_{a+1}\ldots \sigma_{a+c-1} \sigma_{a+c})\ \gamma_w^t \\
%
%
&= (\sigma_{a}\sigma_{a+1}\ldots \sigma_{a+c-1}\ \underline{\sigma_{a+c}})(\sigma_1 \sigma_2 \ldots \sigma_{a+c-2}\sigma_{a+c-1}\sigma_{a+c})(\sigma_{a+c+1}\ldots \sigma_{w-2}\sigma_{w-1})\ \gamma_w^{t-1} \\
&= (\sigma_{a}\sigma_{a+1}\ldots \sigma_{a+c-1})(\sigma_1 \sigma_2 \ldots \sigma_{a+c-2}\ \underline{\sigma_{a+c}\sigma_{a+c-1}\sigma_{a+c}})(\sigma_{a+c+1}\ldots \sigma_{w-2}\sigma_{w-1})\ \gamma_w^{t-1} \\
&= (\sigma_{a}\sigma_{a+1}\ldots \underline{\sigma_{a+c-1}})(\sigma_1 \sigma_2 \ldots \sigma_{a+c-2}\ \sigma_{a+c-1}\sigma_{a+c}\sigma_{a+c-1})(\sigma_{a+c+1}\ldots \sigma_{w-2}\sigma_{w-1})\ \gamma_w^{t-1} \\
&= (\sigma_{a}\sigma_{a+1}\ldots \sigma_{a+c-2})(\sigma_1 \sigma_2 \ldots \underline{\sigma_{a+c-1}\sigma_{a+c-2}\sigma_{a+c-1}}) (\sigma_{a+c}\sigma_{a+c-1})(\sigma_{a+c+1}\ldots \sigma_{w-2}\sigma_{w-1})\ \gamma_w^{t-1} \\
&= (\sigma_{a}\sigma_{a+1}\ldots \sigma_{a+c-2})(\sigma_1 \sigma_2 \ldots \sigma_{a+c-2}\sigma_{a+c-1}\sigma_{a+c-2}) (\sigma_{a+c}\sigma_{a+c-1})(\sigma_{a+c+1}\ldots \sigma_{w-2}\sigma_{w-1})\ \gamma_w^{t-1} \\
&= (\sigma_{a}\sigma_{a+1}\ldots \underline{\sigma_{a+c-2}})(\sigma_1 \sigma_2 \ldots \sigma_{a+c-2})(\sigma_{a+c-1}\sigma_{a+c-2}) (\sigma_{a+c}\sigma_{a+c-1})(\sigma_{a+c+1}\ldots \sigma_{w-2}\sigma_{w-1})\ \gamma_w^{t-1} \\
&= (\sigma_{a}\sigma_{a+1}\ldots \sigma_{a+c-3})(\sigma_1 \ldots \underline{\sigma_{a+c-2} \sigma_{a+c-3} \sigma_{a+c-2}})(\sigma_{a+c-1}\sigma_{a+c-2}) (\sigma_{a+c}\sigma_{a+c-1})(\sigma_{a+c+1}\ldots \sigma_{w-2}\sigma_{w-1})\ \gamma_w^{t-1} \\
&= (\sigma_{a}\sigma_{a+1}\ldots \sigma_{a+c-3})(\sigma_1 \ldots \sigma_{a+c-3})(\sigma_{a+c-2} \sigma_{a+c-3})(\sigma_{a+c-1}\sigma_{a+c-2}) (\sigma_{a+c}\sigma_{a+c-1})(\sigma_{a+c+1}\ldots \sigma_{w-2}\sigma_{w-1})\ \gamma_w^{t-1}. \\
\intertext{We repeat this process -- of moving the last term of the left-most parenthetical as far into the braid as possible using commutation, and then applying the other braid relation -- until we reach $\sigma_a$. Notice that at the end of each iteration of this process, we produce a pair of adjacent terms of the form $(\sigma_{a+c-k} \sigma_{a+c-(k+1)})$. At the penultimate stage, we have:} 
&= (\underline{\sigma_{a}}) (\sigma_1 \sigma_2 \ldots \sigma_{a-2} \sigma_{a-1}) (\sigma_a \sigma_{a+1}\sigma_{a}) (\sigma_{a+2}\sigma_{a+1}) \ldots (\sigma_{a+c}\sigma_{a+c-1})(\sigma_{a+c+1}\ldots \sigma_{w-2}\sigma_{w-1})\ \gamma_w^{t-1} \\
&= (\sigma_1 \sigma_2 \ldots \sigma_{a-2}) (\underline{\sigma_a \sigma_{a-1}\sigma_a})( \sigma_{a+1}\sigma_{a}) (\sigma_{a+2}\sigma_{a+1}) \ldots (\sigma_{a+c}\sigma_{a+c-1})(\sigma_{a+c+1}\ldots \sigma_{w-2}\sigma_{w-1})\ \gamma_w^{t-1} \\
&= (\sigma_1 \sigma_2 \ldots \sigma_{a-2}) (\sigma_{a-1} \sigma_{a}\sigma_{a-1})( \sigma_{a+1}\sigma_{a}) (\sigma_{a+2}\sigma_{a+1}) \ldots (\sigma_{a+c}\sigma_{a+c-1})(\sigma_{a+c+1}\ldots \sigma_{w-2}\sigma_{w-1})\ \gamma_w^{t-1}. \\
\intertext{Reassigning some parenthesis, we obtain:}
&= (\sigma_1 \sigma_2 \ldots \sigma_{a-2}\sigma_{a-1} \sigma_{a})(\sigma_{a-1})( \sigma_{a+1}\sigma_{a}) (\sigma_{a+2}\sigma_{a+1}) \ldots (\sigma_{a+c}\sigma_{a+c-1})(\sigma_{a+c+1}\ldots \sigma_{w-2}\sigma_{w-1})\ \gamma_w^{t-1}. \\
\intertext{We observe that in each parenthetical of the form $(\sigma_{a+c-k} \sigma_{a+c-(k+1)})$, the left term has a larger index than the right term. Moreover, as we read the parentheticals from left to right, the index of the first term uniformly increases until we hit $(\sigma_{a+c+1}\ldots \sigma_{w-2}\sigma_{w-1})\ \gamma_w^t$. Therefore, we can rewrite our braid by collecting terms towards the front of the braid. In the following set of moves, we push the underlined terms to the left:}
&= (\sigma_1 \sigma_2 \ldots \sigma_{a-2}\sigma_{a-1} \sigma_{a})(\sigma_{a-1})(\underline{\sigma_{a+1}}\sigma_{a}) (\sigma_{a+2}\sigma_{a+1}) \ldots (\sigma_{a+c}\sigma_{a+c-1})(\sigma_{a+c+1}\ldots \sigma_{w-2}\sigma_{w-1})\ \gamma_w^{t-1} \\
&= (\sigma_1 \sigma_2 \ldots \sigma_{a-2}\sigma_{a-1} \sigma_{a} \sigma_{a+1})(\sigma_{a-1}\sigma_{a}) (\underline{\sigma_{a+2}}\sigma_{a+1}) \ldots (\sigma_{a+c}\sigma_{a+c-1})(\sigma_{a+c+1}\ldots \sigma_{w-2}\sigma_{w-1})\ \gamma_w^{t-1} \\
&= (\sigma_1 \sigma_2 \ldots \sigma_{a-2}\sigma_{a-1} \sigma_{a} \sigma_{a+1}\sigma_{a+2})(\sigma_{a-1}\sigma_{a}\sigma_{a+1}) \ldots (\sigma_{a+c}\sigma_{a+c-1})(\sigma_{a+c+1}\ldots \sigma_{w-2}\sigma_{w-1})\ \gamma_w^{t-1}. \\
\intertext{Repeating this leftwards operation eventually yields:}
&= (\sigma_1 \sigma_2 \ldots \sigma_{a+1}\sigma_{a+2}\ldots \sigma_{a+c})(\sigma_{a-1}\sigma_{a}\sigma_{a+1}  \ldots \sigma_{a+c-2}\sigma_{a+c-1})(\underline{\sigma_{a+c+1}\ldots \sigma_{w-2}\sigma_{w-1}})\ \gamma_w^{t-1} \\
&= (\sigma_1 \ldots \sigma_{w-1})(\sigma_{a-1}\sigma_{a}\sigma_{a+1}  \ldots \sigma_{a+c-2}\sigma_{a+c-1})\ \gamma_w^{t-1} \\
&= (\sigma_{a-1}\sigma_{a}\sigma_{a+1}  \ldots \sigma_{a+c-2}\sigma_{a+c-1})\ \gamma_w^{t} \\
&= \alpha_2.
\end{align*}
In the last step, we conjugated by $(\sigma_1 \ldots \sigma_{w-1})$. We conclude that $\alpha_1$ and $\alpha_2$ are conjugate.
\end{proof}

\begin{prop} \label{lemma:1BBsAreLorenzKnots}
1-bridge braids are Lorenz knots.
\end{prop}

\begin{proof}
To prove that 1-bridge braids are Lorenz knots, it suffices to show that some sequence of Markov moves transforms $\beta$ to a braid $\tau$, as in \Cref{TLink}.

Let $\beta = (\sigma_b \sigma_{b-1}\ldots \sigma_2 \sigma_1) (\sigma_{w-1}\sigma_{w-2}\ldots \sigma_2 \sigma_1)^t$ denote the standard braid presentation of a 1-bridge braid. Let $\beta' = (\sigma_{w-b}\sigma_{w-b+1}\ldots \sigma_{w-1})(\sigma_1 \sigma_2 \ldots \sigma_{w-1})^t$. We claim that $\widehat{\beta}$ and $\widehat{\beta'}$ are isotopic knots in $S^3$: view $S^3$ as $\mathbb{R}^3 \cup \{\infty\}$, and fix the circle $C = z$-axis $\cup \{\infty\}$; we represent the $z$-axis by the purple dotted line in \Cref{fig:Reflect}. We draw the braid $\beta$ on the ``left'' side of $C$, and then rotate $\beta$ about the purple line; this produces $\beta'$, which is seen on the ``right'' side of $C$. In particular, if we take $\widehat{\beta}$ and follow it through the rotation isotopy, we will get $\widehat{\beta'}$. Therefore, $\widehat{\beta} = \widehat{\beta'}$ as knots in $S^3$. Alternatively, one can use some standard results in braid theory: if we conjugate $\beta$ by the Garside element $\Delta \in B_w$, we produce $\beta'$ (see \cite{Garside, GonzalesMeneses:BraidGroups} for more details); since conjugation preserves the link type of the closure, $\widehat{\beta}$ and $\widehat{\beta'}$ present the same knot.

\begin{figure}[htb]
\centering
\begin{tikzpicture}
    \draw (1, 0) node[inner sep=0] {\includegraphics[scale=0.11]{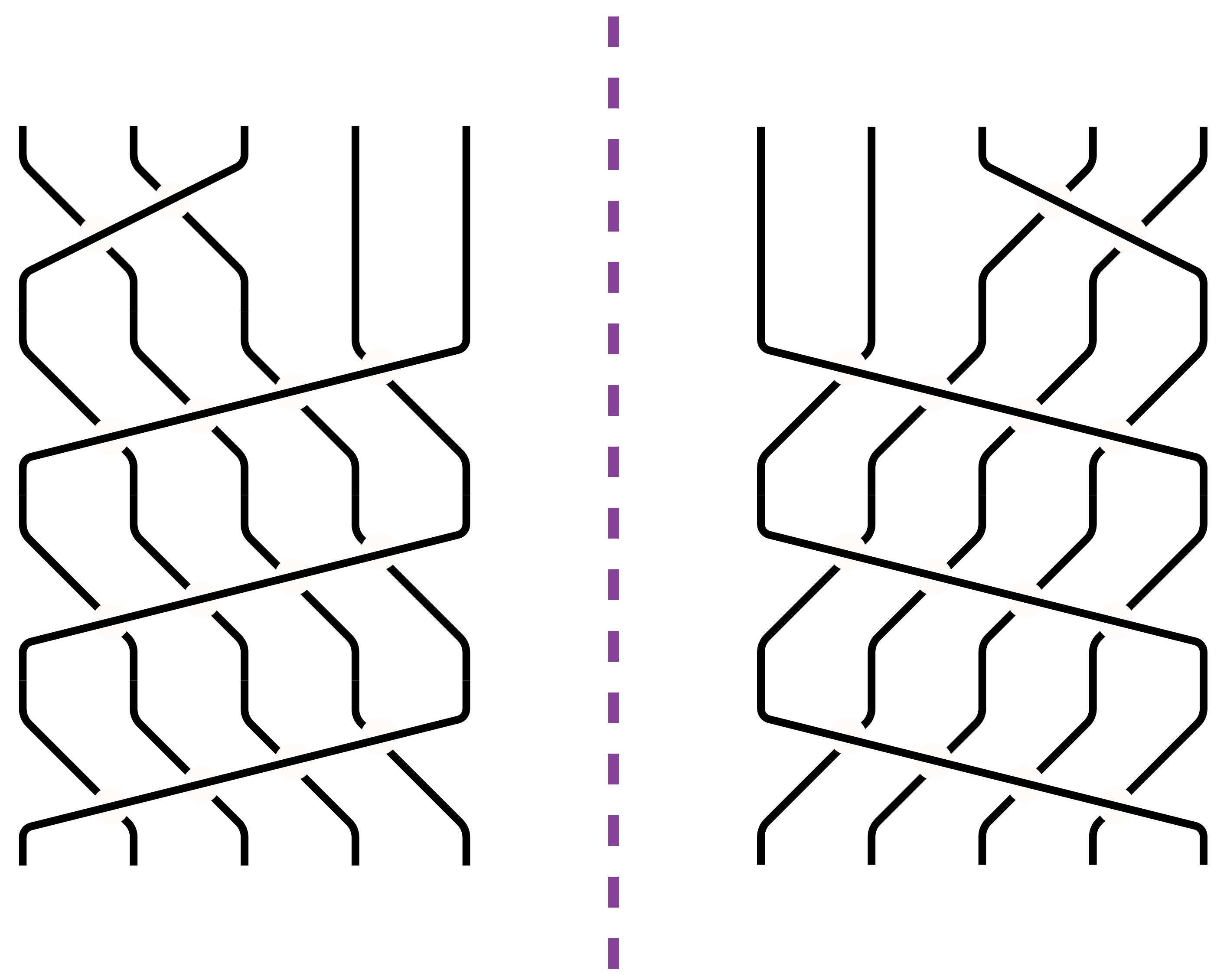}};
\end{tikzpicture}
\captionof{figure}{Rotating the left braid about the purple line produces the right braid.}
\label{fig:Reflect}
\end{figure}

Next, we perform $(w-b-1)$ many applications of \Cref{lemma:LeftMove}:
\begin{align*}
\beta &= (\sigma_b \sigma_{b-1}\ldots \sigma_2 \sigma_1) (\sigma_{w-1}\sigma_{w-2}\ldots \sigma_2 \sigma_1)^t \\
& = (\sigma_{w-b}\sigma_{w-b+1}\ldots \sigma_{w-1})(\sigma_1 \sigma_2 \ldots \sigma_{w-1})^t \\
& = (\sigma_{w-b-1}\sigma_{w-b}\ldots \sigma_{w-2})(\sigma_1 \sigma_2 \ldots \sigma_{w-1})^t \\
& = (\sigma_{1}\sigma_{2}\ldots \sigma_{b})(\sigma_1 \sigma_2 \ldots \sigma_{w-1})^t.
\end{align*}
Thus, the 1-bridge braid $\beta$ admits a T-link presentation.
\end{proof}

\begin{lemma} \label{lemma:ExamplesOfLorenzKnots}
Twisted torus knots and $n$-bridge braids are Lorenz knots.
\end{lemma}

\begin{proof}
Twisted torus knots are the closures of positive braids on $w$ strands with the following form: \[\rho = (\sigma_{w-1} \sigma_{w-2} \ldots \sigma_1)^t (\sigma_{w-1} \sigma_{w-2} \ldots \sigma_{w-k})^{sk}.\] Rotating $\rho$ as in \Cref{fig:Reflect} yields $\rho' = (\sigma_1 \sigma_2 \ldots \sigma_{w-1})^t (\sigma_1 \sigma_2 \ldots \sigma_k)^{sk}$. We know that $\widehat{\rho}$ and $\widehat{\rho'}$ are isotopic knots; since $\rho'$ is presented as a T-link braid, we deduce that twisted torus knots are T-links.

Indeed, the braided presentation for $n$-bridge braids appears very similar to those of twisted torus knots (however, there is not required that $b$ divides $n$). We quickly show that these, too, are T-links:
\begin{align*}
\eta &= (\sigma_{b} \sigma_{b-1} \ldots \sigma_{1})^{n}(\sigma_{w-1} \sigma_{w-2} \ldots \sigma_1)^t \\
&= (\sigma_{w-b} \sigma_{w-b+1} \ldots \sigma_{w-1})^{n}(\sigma_{1} \sigma_{2} \ldots \sigma_{w-1})^t \\
&=(\sigma_{w-b} \sigma_{w-b+1} \ldots \sigma_{w-1})^{n-1}(\sigma_{w-b} \sigma_{w-b+1} \ldots \sigma_{w-1})(\underline{\sigma_{1} \sigma_{2} \ldots \sigma_{w-1}})(\sigma_{1} \sigma_{2} \ldots \sigma_{w-1})^{t-1}. \\
\intertext{In the proof of \Cref{lemma:LeftMove}, we only performed braid relationships -- the only place we conjugated our braid is in the last step. Thus, applying the proof of \Cref{lemma:LeftMove}, we see:}
&=(\sigma_{w-b} \sigma_{w-b+1} \ldots \sigma_{w-1})^{n-1} (\underline{\sigma_{1} \sigma_{2} \ldots \sigma_{w-1}})(\sigma_{w-b-1} \sigma_{w-b} \ldots \sigma_{w-2}) (\sigma_{1} \sigma_{2} \ldots \sigma_{w-1})^{t-1} \\
&=(\sigma_{1} \sigma_{2} \ldots \sigma_{w-1})(\sigma_{w-b-1} \sigma_{w-b} \ldots \sigma_{w-2})^n (\sigma_{1} \sigma_{2} \ldots \sigma_{w-1})^{t-1} \\
&= (\sigma_{w-b-1} \sigma_{w-b} \ldots \sigma_{w-2})^n (\sigma_{1} \sigma_{2} \ldots \sigma_{w-1})^{t}. \\
\intertext{We repeat this process an additional $w-b-2$ times, yielding:}
&= (\sigma_{1} \sigma_{2} \ldots \sigma_{b})^n (\sigma_{1} \sigma_{2} \ldots \sigma_{w-1})^{t}.
\end{align*}
Thus, $n$-bridge braids are more general than twisted torus knots, and they are T-links.
\end{proof}

\section{Preliminaries for the proof of the main theorem}
\label{section:preliminaries}

\begin{rmk*}
After posting this article to the arXiv, the author of \cite{Nie:1BBsSatellites} informed us that the results in this section are sketched in the body of the proof of Theorem 1.1 in \cite{Nie:1BBsSatellites}. This work contains the full proofs.
\end{rmk*}

\begin{defn} \label{def2}
We define $\delta_n$ and $\gamma_n$ to be the positive braid words $\delta_n := (\sigma_n \ldots \sigma_2 \sigma_1)$ and $\gamma_n := (\sigma_1 \sigma_2 \ldots \sigma_n)$ in $B_{r}$, the braid group on $r$ strands, where $r \geq n+1$. 
\end{defn}

\begin{rmk}
Note that $\mathcal{K}(w, b,t) = \widehat{\delta_b \delta_{w-1}^t}$. 
\end{rmk}

\begin{lemma} \label{lem3}
Let $\delta_k \in B_{k+1}$. Then $\delta_k \ \delta_k = \delta_{k-1}\ \delta_k \ \sigma_1$ as braid words in $B_{k+1}$. 
\end{lemma}

\begin{proof}
We begin by expanding the left hand side.
\begin{align}
\delta_k \ \delta_k &= (\sigma_k \  \sigma_{k-1} \ \sigma_{k-2} \ \dots \ \sigma_3 \ \sigma_2 \  \sigma_1)(\underline{\sigma_k} \dots \sigma_1) \\
&= (\underline{\sigma_k \  \sigma_{k-1} \ \sigma_k} \ \sigma_{k-2} \ \dots \  \sigma_3 \ \sigma_2 \  \sigma_1)(\sigma_{k-1} \dots \sigma_1) \\
&= (\sigma_{k-1} \  \sigma_{k} \ \sigma_{k-1} \ \sigma_{k-2} \ \sigma_{k-3}\ \dots \  \sigma_3 \ \sigma_2 \  \sigma_1)(\underline{\sigma_{k-1}} \dots \sigma_1) \\
&= (\sigma_{k-1} \  \sigma_{k} \ \underline{\sigma_{k-1} \ \sigma_{k-2} \ \sigma_{k-1}} \ \sigma_{k-3}\ \dots \  \sigma_3 \ \sigma_2 \  \sigma_1)(\sigma_{k-2} \dots \sigma_1) \\
&= (\sigma_{k-1} \  \sigma_{k} \ \underline{\sigma_{k-2}} \ \sigma_{k-1} \ \sigma_{k-2} \ \sigma_{k-3}\ \dots \  \sigma_3 \ \sigma_2 \  \sigma_1)(\sigma_{k-2} \dots \sigma_1) \\
&= (\sigma_{k-1} \  \sigma_{k-2} \ \sigma_{k} \ \sigma_{k-1} \ \sigma_{k-2} \ \sigma_{k-3}\ \dots \  \sigma_3 \ \sigma_2 \  \sigma_1)(\underline{\sigma_{k-2}} \dots \sigma_1) \\
&= (\sigma_{k-1} \  \sigma_{k-2} \ \sigma_{k} \ \sigma_{k-1} \ \underline{\sigma_{k-2} \ \sigma_{k-3}\ \sigma_{k-2}}\dots \  \sigma_3 \ \sigma_2 \  \sigma_1)(\sigma_{k-3} \dots \sigma_1) \\
&= (\sigma_{k-1} \  \sigma_{k-2} \ \sigma_{k} \ \sigma_{k-1} \ \underline{\sigma_{k-3}} \ \sigma_{k-2}\ \sigma_{k-3}\dots \  \sigma_3 \ \sigma_2 \  \sigma_1)(\sigma_{k-3} \dots \sigma_1) \\
&= (\sigma_{k-1} \  \sigma_{k-2} \ \sigma_{k-3} \ \sigma_{k} \ \sigma_{k-1}  \ \sigma_{k-2}\ \sigma_{k-3}\dots \  \sigma_3 \ \sigma_2 \  \sigma_1)(\sigma_{k-3} \dots \sigma_1).
\end{align}

We describe the operations at play: in line (3), we identify the $\sigma_k$ letter that is furthest to the right, and apply $k-2$ commuting relations to push it as much to the left as possible. This creates the underlined subword  $\sigma_{k}\sigma_{k-1} \sigma_{k}$ in line (4); applying the braid relation yields line (5). We repeat this procedure (of finding the largest letter in the right parenthetical subword, applying commuting relations to push it as far to the left as possible, and then applying a braid relation in lines (5), (6), and (7). From lines (7) to (8), we identify and execute another commuting relation. We call this 4-step procedure a \emph{left push}, and say that we perfom a \emph{left push on $\sigma_{t}$} when $\sigma_t$ is the largest letter in the second parenthetical braid word. Notice that after executing the \text{left push} operation on $\sigma_{r+1}$, the braid word decomposes into the three subwords $((\sigma_{k-1} \ldots \sigma_{r})(\sigma_{k} \ldots \sigma_{1}))(\sigma_r \ldots \sigma_1)$; this is seen explicitly in lines (8) and (11). 
After repeating the \emph{left push} operation another another $k-5$ times from line (11) onwards, we get:

\begin{align}
&= ((\sigma_{k-1} \  \sigma_{k-2} \ \sigma_{k-3} \ldots \sigma_{3} \ \sigma_{2}) \ (\sigma_{k} \ \sigma_{k-1}  \ \sigma_{k-2}\ \sigma_{k-3}\dots \  \sigma_3 \ \sigma_2 \  \sigma_1))(\sigma_{2} \ \sigma_1) \\
&= (\sigma_{k-1} \  \sigma_{k-2} \ \sigma_{k-3} \ldots \sigma_{3} \ \sigma_{2} \ \  \sigma_{k} \ \sigma_{k-1}  \ \sigma_{k-2}\ \sigma_{k-3}\dots \  \sigma_3 \ \underline{\sigma_2 \  \sigma_1 \ \sigma_{2}} \ \sigma_1) \\
&= (\sigma_{k-1} \  \sigma_{k-2} \ \sigma_{k-3} \ldots \sigma_{3} \ \sigma_{2} \ \  \sigma_{k} \ \sigma_{k-1}  \ \sigma_{k-2}\ \sigma_{k-3}\dots \  \sigma_3 \ \underline{\sigma_1} \  \sigma_2 \ \sigma_{1} \ \sigma_1) \\
&= ((\sigma_{k-1} \  \sigma_{k-2} \ \sigma_{k-3} \ldots \sigma_{3} \ \sigma_{2} \ \sigma_1) \ (\sigma_{k} \ \sigma_{k-1}  \ \sigma_{k-2}\ \sigma_{k-3}\dots \  \sigma_3  \  \sigma_2 \ \sigma_{1})) \ \sigma_1 \\
&= \delta_{k-1} \ \delta_{k} \ \sigma_1.
\end{align}
This is exactly what we wanted to show. 
\end{proof}

\begin{lemma} \label{lem4}
Let $\delta_k \in B_{k+1}$. Then $\sigma_1 \ \delta_k = \delta_k \ \sigma_2.$
\end{lemma}

\begin{proof}
We begin by expanding the left hand side.
\begin{align*}
\sigma_1 \ \delta_k &= \sigma_1 \  \sigma_k \  \sigma_{k-1} \ \sigma_{k-2} \dots \sigma_3 \ \sigma_2 \  \sigma_1 \\
&= \underline{\sigma_1} \  \sigma_k \ \sigma_{k-1} \ \sigma_{k-2} \ \dots \sigma_3 \ \sigma_2 \ \sigma_1 \\
&= \sigma_k \ \sigma_{k-1} \ \sigma_{k-2} \dots \sigma_3 \ \underline{\sigma_1 \ \sigma_2 \ \sigma_1} \\
&= (\sigma_k \ \sigma_{k-1} \ \sigma_{k-2} \dots \sigma_3 \ \sigma_2 \ \sigma_1) \ \sigma_2 \\
&= \delta_k \sigma_2.
\end{align*}
This is what we wanted to prove.
\end{proof}

\begin{lemma} \label{lem5}
Let $\delta_k \in B_{k+1}$. Then $\sigma_j \ \delta_k = \delta_k \ \sigma_{j+1}$ when $1 < j < k$.
\end{lemma}

\begin{proof}
Suppose $1<j<k$. We begin by expanding $\sigma_j \delta_k$:
\begin{align*}
\sigma_j \delta_k 
&= \underline{\sigma_j}(\sigma_k \sigma_{k-1} \dots \sigma_{j+2} \  \sigma_{j+1} \sigma_j {\sigma_{j-1}} \ \dots \sigma_1) \\
&= \sigma_k \sigma_{k-1} \dots \sigma_{j+2} \ \underline{\sigma_j \sigma_{j+1} \sigma_j} {\sigma_{j-1}} \ \dots \sigma_1 \\
&= \sigma_k \sigma_{k-1} \dots \sigma_{j+2} \ \sigma_{j+1} \sigma_{j} \underline{\sigma_{j+1}} {\sigma_{j-1}} \ \dots \sigma_1 \\
&= (\sigma_k \sigma_{k-1} \dots \sigma_{j+2} \ \sigma_{j+1} \sigma_{j} {\sigma_{j-1}} \ \dots \sigma_1) \sigma_{j+1}
\end{align*}
which is $\delta_k \sigma_{j+1}$, as desired. 
\end{proof}

\begin{lemma} \label{lem6}
Let $\delta_k  \in B_{k+1}$.  When $s < k$, we have
$\sigma_1 \ \delta_k^s = \delta_k^s \ \sigma_{s+1}$.
\end{lemma}

\begin{proof}
We see that
\begin{align*}
     \sigma_1 {\delta_k}^s &= \sigma_1 \delta_k {\delta_k}^{s-1}\\
    &= \delta_k \sigma_2 \delta_k {\delta_k}^{s-2} \text{ by \Cref{lem4}, } \\
    &= {\delta_k}^2 \sigma_3 \delta_k {\delta_k}^{s-3} \text{ by \Cref{lem5}.}  
    \end{align*}
Applying \Cref{lem5} a total of $s-1$ times, we get: 
\begin{align*}
    \delta_k^s \ \sigma_{s+1}. 
\end{align*}
Thus, $\sigma_1 \ \delta_k^s = \delta_k^s \ \sigma_{s+1}$.
\end{proof}

Note that \Cref{lem5} generalizes \Cref{lem4}, and \Cref{lem6} combines Lemmas~\ref{lem4} and \ref{lem5}.

\begin{prop} \label{prop8}
Let $\delta_j \in B_{j+1}$. Then 
$\delta_j^t = \delta_{j-1} \ \delta_j^{t-1} \ \sigma_{t-1}$,
where $t<j$.
\begin{proof}
We see that
\begin{align*}
    \delta_j^t &= (\delta_j \delta_j) \delta_j^{t-2}\\
    &= (\delta_{j-1}\delta_j\sigma_1)\delta_j^{t-2} \text{ by \Cref{lem3} }\\
    &=\delta_{j-1}\delta_j(\sigma_1 \delta_j^{t-2})\\
    &= \delta_{j-1}\delta_j(\delta_j^{t-2}\sigma_{(t-2)+1}) \text{ by \Cref{lem6}}\\
    &= \delta_{j-1}\delta_j\delta_j^{t-2}\sigma_{t-1}\\
    &= \delta_{j-1}\delta_j^{t-1}\sigma_{t-1}.
\end{align*}
Thus, $\delta_j^t = \delta_{j-1} \ \delta_j^{t-1} \ \sigma_{t-1}$.
\end{proof}
\end{prop}

\begin{prop} \label{prop9}
Let $\delta_k, \gamma_{t-1} \in B_{k+1}$. Then, 
$\delta_k^t = \delta_{k-1}^{t-1} \ \delta_k \  \gamma_{t-1}$
where $t<k$.
\end{prop}

\begin{proof}
Note that
\begin{align*}
    \delta_k^t &= \delta_{k-1} ({\delta_k}^{t-1}) \sigma_{t-1} \text{ by \Cref{prop8}} \\
    &= \delta_{k-1} (\delta_{k-1} {\delta_k}^{t-2} \sigma_{t-2}) \sigma_{t-1} \text{ by \Cref{prop8}} \\
    &= \delta_{k-1}^2 \ ({\delta_k}^{t-2}) \  \sigma_{t-2} \sigma_{t-1}. 
    \intertext{Iteratively apply \Cref{prop8} an additional $t-3$ times to the rightmost $\delta_{k}^{t - \star}$ to obtain}
  \delta_k^t  &= \delta_{k-1}^{t-1} \ \delta_k^1 \ (\sigma_1 \sigma_2 \ldots \sigma_{t-1}) \\
  &= \delta_{k-1}^{t-1} \ \delta_k \ \gamma_{t-1}.
\end{align*} 
This yields the desired conclusion.
\end{proof}

\begin{prop} \label{prop:SameSuperSub}
Let $\delta_k \in$ $B_{k+1}$. Then,
$\delta_k^k = \delta_{k-1} (\delta_k^{k-1})\sigma_{k-1}$.
\end{prop}

\begin{proof}
\begin{align*}
\delta_k^k
&= (\delta_k^{k-1}) \delta_k \\
 &= (\delta_{k-1} \delta_k ^{k-2} \sigma_{k-2}) \delta_k  \text{\qquad by \Cref{prop8}}\\
 &= \delta_{k-1} \delta_k ^{k-2} (\sigma_{k-2} \delta_k)  \\
 &= \delta_{k-1} \delta_k^{k-2}(\delta_k \sigma_{k-1})  \text{\qquad by \Cref{lem5}} \\
 &= \delta_{k-1} (\delta_k^{k-1})\sigma_{k-1}.
\end{align*}
This is what we wanted to show.
\end{proof}

\begin{prop} \label{2023prop}
Let $\delta_k \in B_{k+1}$, and suppose $\alpha \in B_{k}$ $($so, in particular, $\alpha$ is a braid word on strictly fewer strands than $\delta_k)$. Then $\alpha \ \delta_k^k = \alpha \ \delta_{k-1}^{k} \gamma_{k-1}$.
\end{prop}

\begin{proof}
The proof requires straightforward applications of \Cref{prop9} and \Cref{prop:SameSuperSub}.
\begin{align*}
    \alpha \ \delta_k^k 
    &= \alpha \ (\delta_{k-1} \ \delta_{k}^{k-1} \ \sigma_{k-1}) \text{ \qquad by \Cref{prop:SameSuperSub}} \\
    &= \alpha \ \delta_{k-1} \ (\delta_{k-1}^{k-2} \ \delta_{k} \ \gamma_{k-2})\ \sigma_{k-1} \text{ \qquad by \Cref{prop9}} \\
    &= \alpha \ \delta_{k-1}^{k-1} \ \delta_{k} \ \gamma_{k-1} \\
    &= \alpha \ \delta_{k-1}^{k-1} \ \sigma_k \ \delta_{k-1} \ \gamma_{k-1} \\
\intertext{Since $\alpha$ is a braid in $B_k$, there is a unique $\sigma_k$ letter in this braid. So, we can destabilize to get:}
    &= \alpha \ \delta_{k-1}^{k} \ \gamma_{k-1}\\
\end{align*}
This is what we wanted to show.
\end{proof}

\section{Proof of the main theorem}
\label{section:n-bridge}

\noindent \textbf{\Cref{thm:NBridge}.} \textit{The braid index of an n--bridge braid $$\mathcal{K}(w,b,t,n)$$ is:}
\begin{center}
$ i(\mathcal{K}(w,b,t,n)) = \begin{cases} 
      w &  t \geq w, \ n\geq 1 \\
      t &  w > t > b, \ n \geq 1 \\
      t+1 & w > b \geq t, n=1\\
      b+1 & w > b\geq t, n+t \geq b+1, n>1\\
      n+t & w > b \geq t, n+t < b+1, n>1.
      \end{cases} $
\end{center}

\begin{proof} 

We use \Cref{prop8} and \Cref{prop9} and destabilizations to find a presentation of the knot which allows us to apply \Cref{thm:FMW}. \\

\noindent \fbox{\textbf{Case 1: $t \geq w, n \geq 1$}} \\

\noindent Let $t \geq w$ and $n \geq 1$. Then, we know:
\begin{align*}
    \mathcal{K}(w,b,t,n)
    &= \delta_b^n \delta_{w-1}^t \\
    &= \delta_b^n \delta_{w-1}^w \delta_{w-1}^{t-w}.
\end{align*}
Since $(\sigma_{w-1} \sigma_{w-2} \ldots \sigma_1)^w$ is a full twist on $w$ strands, by applying   \Cref{thm:FMW}, $i(\mathcal{K}(w,b,t,n)) = w$.
\bigbreak \noindent

\noindent \fbox{\textbf{Case 2: $w > t > b, n \geq 1$}} \\

\noindent Suppose $w > t > b$ and $n \geq 1$. If $w-1 = t > b$, then we have:
\begin{align*}
    \mathcal{K}(w,b,t,n) &= \delta_b^n \delta_{w-1}^t\\
    &= \delta_b^n \delta_{w-1}^{w-1} \\
    &= \delta_b^n \ \delta_{w-2}^{w-1} \ \gamma_{w-2} \text{ \qquad by \Cref{2023prop}}\\
    &= \delta_b^n \ \delta_{t-1}^{t} \ \gamma_{t-1} \text{ \qquad since $t = w-1$}
\end{align*}
Therefore, the braid can be written to contain $\delta_{t-1}^t$, which is a full twist on $t$ strands. By \Cref{thm:FMW}, $i(\mathcal{K}(w,b,t,n)) = t$ if $w-1 = t > b, n \geq 1$. \\

\noindent We now study what happens if $w-1 > t > b$. We know:
\begin{align*}
    \mathcal{K}(w,b,t,n) &= \delta_b^n \underline{\delta_{w-1}^t}\\
    &= \delta_b^n \underline{\delta_{w-2}^{t-1} \delta_{w-1} \gamma_{t-1}} \text{ by \Cref{prop9}} \\
    &= \delta_b^n \delta_{w-2}^{t-1} \underline{\sigma_{w-1} \delta_{w-2}} \gamma_{t-1}.
\end{align*}
Since $w > t > b$, then $w-1 \geq t > b$, hence there is a single $\sigma_{w-1}$ in the braid word, which is currently in $B_{w-1}$. Thus, we can destablize the braid to produce a new braid in $B_{w-2}$:
\begin{align*}
\mathcal{K}(w,b,t,n) &= \delta_b^n \delta_{w-2}^{t-1} \underline{\sigma_{w-1} \delta_{w-2}} \gamma_{t-1} \\
&= \delta_b^n \delta_{w-2}^{t-1} \underline{\delta_{w-2}} \gamma_{t-1} \\
&= \delta_b^n \delta_{w-2}^{t} \gamma_{t-1}.
\end{align*}

\noindent We iteratively: (1) apply \Cref{prop9} to the rightmost $\delta_{w-\star}^t$ term, and (2) destabilize the largest remaining Artin generator. Since $w > t > b$ we can repeat the above process a total of $w - t$ times, after which we have:
\begin{align*}
\mathcal{K}(w, b,t) &= \delta_b^n \delta_{t-1}^t \gamma_{t-1}^{w-t}.
\end{align*}
As $t > b$, we know that $t-1 \geq b$, hence $\delta_{b}^n$ contains no $\sigma_s$ letters, where $s \geq t-1$. Moreover, this is a braid word in $B_t$, and it contains a full twist on $t$ strands. By \Cref{thm:FMW}, $i(\mathcal{K}(w,b,t,n)) = t$. 
\bigbreak \noindent

\noindent \fbox{\textbf{Case 3: $w > b \geq t, n=1$}} \\

\noindent Our definition of a 1-bridge braid requires that $w-2 \geq b$, so we may revise our assumptions to be $w-2 \geq b \geq t$. In particular, note that $t < w-1$. We have:
\begin{align*}
    \mathcal{K}(w, b,t) &= \delta_b \underline{\delta_{w-1}^t} \\
    &= \delta_b \underline{\delta_{w-2}^{t-1} \delta_{w-1} \gamma_{t-1}} \text{ by \Cref{prop9}} \\ 
    &= \delta_b \delta_{w-2}^{t-1} \underline{\sigma_{w-1} \delta_{w-2}} \gamma_{t-1}.
\end{align*}
Since $ w-2 \geq b$, there is a single $\sigma_{w-1}$, and we can destabilize the braid:
\begin{align*}
\mathcal{K}(w, b,t) &= \delta_b \delta_{w-2}^{t-1} \underline{\sigma_{w-1}} \delta_{w-2} \gamma_{t-1} \\
&= \delta_b \delta_{w-2}^{t-1} \delta_{w-2} \gamma_{t-1}\\
&= \delta_b \delta_{w-2}^{t} \gamma_{t-1}.
\end{align*}
We iteratively: (1) apply \Cref{prop9} to the rightmost $\delta_{w-\star}^t$ term, and (2) destabilize the largest remaining Artin generator. 
Since $w > b \geq t$, we can repeat the above process a total of $w - b-1$ times, after which we have:
\begin{align*}
\mathcal{K}(w, b,t) &= \delta_b \underline{\delta_{w-2}^{t} \gamma_{t-1}} \\
&= \delta_b \underline{\delta_{b}^t \gamma_{t-1}^{w-b-1}}. 
\end{align*}
This braid word is in $B_{b+1}$, the braid group on $b+1$ strands.
If $b=t$, then: \[\mathcal{K}(w, b,t) = \delta_b^{t+1}\gamma_{t-1}^{w-b-1} = \delta_t^{t+1}\gamma_{t-1}^{w-b-1}.\]
Applying \Cref{thm:FMW} allows us to conclude that $i(\mathcal{K}(w, b,t))=t+1$. \\

Otherwise, $b>t$ and we iteratively: (1) apply \Cref{prop9} to the $\delta_{b - \star}^{t+1}$ term (note that $\star = 0$ to start), and (2) destabilize the largest remaining Artin generator. Since $b > t$, we can repeat this process a total of $b-t$ times to obtain:
\begin{align*}
\mathcal{K}(w, b,t) &= \delta_b^{t+1} \gamma_{t-1}^{w-b-1} \\
& = \delta_{t}^{t+1} \gamma_{t}^{b-t} \gamma_{t-1}^{w-b-1}.
\end{align*}
Once again, by \Cref{thm:FMW}, we deduce $i(\mathcal{K}(w, b,t))=t+1$. \\

\noindent \fbox{\textbf{Case 4: $w > b\geq t, n+t \geq b+1, n>1$}} \\

\noindent Suppose $w > b\geq t, n+t \geq b+1,$ and $n>1.$ Our definition of a 1-bridge braid requires that $w-2 \geq b$, so we may revise our assumptions to be $w-2 \geq b \geq t$. In particular, note that $t < w-1$. We have:
\begin{align*}
    \mathcal{K}(w,b,t,n) &= \delta_b^n (\delta_{w-1}^t)\\
    &= \delta_b^n \delta_{w-2}^{t-1} (\delta_{w-1}) \gamma_{t-1} \text{ by \Cref{prop9}} \\
    &= \delta_b^n \delta_{w-2}^{t-1}(\sigma_{w-1} \delta_{w-2}) \gamma_{t-1}
\end{align*}
Since $w-2 \geq b$, there is a single $\sigma_{w-1}$, and we can destabilize the braid:
\begin{align*}
\mathcal{K}(w, b,t) &= \delta_b^n \delta_{w-2}^{t-1} (\sigma_{w-1}) \delta_{w-2} \gamma_{t-1} \\
&= \delta_b^n \delta_{w-2}^{t-1} \delta_{w-2} \gamma_{t-1}\\
&= \delta_b^n \delta_{w-2}^{t} \gamma_{t-1}.
\end{align*} 

\noindent We iteratively: (1) apply \Cref{prop9} to the rightmost $\delta_{w - \star}^t$ term, and (2) destabilize the largest remaining Artin generator. Since $w > b \geq t$, we can repeat the above process a total of $w - b-1$ times, after which we have:
\begin{align*}
\mathcal{K}(w, b,t) &= \delta_b^n \delta_{w-2}^t \gamma_{t-1} \\
&= \delta_b^n \delta_{b}^{t} \gamma_{t-1}^{w-b-1}.
\end{align*}
This is a braid word on $b+1$ strands.
As $n+t \geq  b+1$, we get 
\begin{align*}
\mathcal{K}(w,b,t) &= \delta_b^n \delta_b^t \gamma_{t-1}^{w-b-1} \\
&= \delta_b^{n+t} \gamma_{t-1}^{w-b-1}.
\end{align*}
Applying \Cref{thm:FMW}, we deduce $i(\mathcal{K}(w,b,t,n))=b+1$. \\

\noindent \fbox{\textbf{Case 5: $w > b \geq t, n+t < b+1$}} \\

\noindent Our definition of a 1-bridge braid requires that $w-2 \geq b$, so we may revise our assumptions to be $w-2 \geq b \geq t$. In particular, $t < w-1$. We begin by applying \Cref{prop9} to the standard braided presentation of $\mathcal{K}(w,b,t,n)$:
\begin{align*}
    \mathcal{K}(w,b,t,n) &= \delta_b^n \underline{\delta_{w-1}^t}\\
    &= \delta_b^n \underline{\delta_{w-2}^{t-1} (\delta_{w-1}) \gamma_{t-1}} \text{ by \Cref{prop9}} \\
    &= \delta_b^n \delta_{w-2}^{t-1} \underline{\sigma_{w-1} \delta_{w-2}} \gamma_{t-1}.
\end{align*}
Our definition of $n$-bridge braid required that $ w-2 \geq b$. Therefore, there is a single $\sigma_{w-1}$, and we can destabilize the braid:
\begin{align*}
\mathcal{K}(w,b,t,n) &= \delta_b^n \delta_{w-2}^{t-1} \underline{\sigma_{w-1}} \delta_{w-2} \gamma_{t-1} \\
&= \delta_b^n \delta_{w-2}^{t-1} \delta_{w-2} \gamma_{t-1}\\
&= \delta_b^n \delta_{w-2}^{t} \gamma_{t-1}.
\end{align*} 

\noindent We iteratively: (1) apply \Cref{prop9} to the rightmost $\delta_{w - \star}^t$ term, and (2) destabilize the largest remaining Artin generator. Since $w > b \geq t$, we can repeat the above process a total of $w - b-1$ times, after which we have:
\begin{align}
\mathcal{K}(w,b,t,n) &= \delta_b^n \delta_{w-2}^{t} \gamma_{t-1} \nonumber\\
    &= \delta_b ^{n} \delta_b^t \gamma_{t-1}^{w-b-1} \label{eqn:remember}.
\end{align}
This braid word is on $b+1$ strands. We assumed that $n+t < b+1$, so namely, $n+t \leq b$. Suppose $n+t = b$. In this case, 
\begin{align*}
\mathcal{K}(w,b,t,n) &= \delta_b ^{n} \delta_b^t \gamma_{t-1}^{w-b-1} \\
&= \delta_b^b \gamma_{t-1}^{w-b-1}  \\
&= \delta_{b-1} \underline{\delta_b^{b-1}}\sigma_{b-1} \gamma_{t-1}^{w-b+1} \text{\qquad by \Cref{prop:SameSuperSub}} \\
&= \delta_{b-1} \underline{\delta_{b-1} ^{b-2} \delta_b \gamma_{b-2}} \sigma_{b-1} \gamma_{t-1}^{w-b+1} \text{\qquad by \Cref{prop9}} \\
&= \delta_{b-1}\delta_{b-1} ^{b-2} \underline{\delta_b} \gamma_{b-2} \sigma_{b-1} \gamma_{t-1}^{w-b+1} \\
&= \delta_{b-1}\delta_{b-1} ^{b-2} \underline{\sigma_b \delta_{b-1}} \gamma_{b-2} \sigma_{b-1} \gamma_{t-1}^{w-b+1} \\
&= \delta_{b-1}\delta_{b-1} ^{b-1} \gamma_{b-2} \sigma_{b-1} \gamma_{t-1}^{w-b+1} \text{\qquad by destabilizing the unique $\sigma_b$ letter} \\
&= \delta_{b-1} ^{b} \gamma_{b-1} \gamma_{t-1}^{w-b+1}.
\end{align*}

\noindent This braid word on $b$ strands contains a full twist; thus, by \Cref{thm:FMW}, $i(\mathcal{K}(w,b,t,n))= b = n + t$. Now suppose $n+t < b $. In particular, $n+t \leq b-1$. In this case, as in \Cref{eqn:remember},

\begin{align*}
    K(w, b, t, n) &= \delta_b ^{n} \delta_b^t \gamma_{t-1}^{w-b-1}\\
    &=\underline{\delta_b ^{n+t}} \gamma_{t-1}^{w-b-1}\\
    &= \underline{\delta_{b-1}^{n+t-1} \delta_b \gamma_{n + t - 1}} \gamma_{t-1}^{w-b-1} \text{\qquad by \Cref{prop9}} \\
 &= \delta_{b-1}^{n+t-1} \underline{\delta_b} \gamma_{n + t - 1} \gamma_{t-1}^{w-b-1}  \\
  &= \delta_{b-1}^{n+t-1} \underline{\sigma_b \delta_{b-1}} \gamma_{n + t - 1} \gamma_{t-1}^{w-b-1}  \\
&= \delta_{b-1}^{n+t-1} \delta_{b-1} \gamma_{n + t - 1} \gamma_{t-1}^{w-b-1} \text{\qquad by destabilizing the unique $\sigma_b$ letter} \\
&= \delta_{b-1}^{n+t} \gamma_{n + t - 1} \gamma_{t-1}^{w-b-1}. 
\end{align*}

\noindent To simplify the right hand side, we will need to (1) apply \Cref{prop9} to the $\delta_{b - \star}^{n+t}$ term, and then (2) destabilize the largest remaining Artin generator. We will need to repeat this process $(b-1) - (n+t)$ many times. Below, we write out explicitly what happens after applying steps (1) and (2) once, and then suppress the word for the remaining $(b-1)-(n+t)-1$ applications \Cref{prop9} and destabilization. We note: implicitly, we really are using that $n+t \leq b-1$.

\begin{align*}
K(w, b, t, n) &= \delta_{b-1}^{n+t} \gamma_{n + t - 1} \gamma_{t-1}^{w-b-1} \\
&= (\delta_{b-2}^{n+t-1} \delta_{b-1} \gamma_{n+t-1}) \gamma_{n + t - 1} \gamma_{t-1}^{w-b-1} \text{\qquad by \Cref{prop9}}\\
&= \delta_{b-2}^{n+t-1} \sigma_{b-1} \delta_{b-2} \gamma_{n+t-1}^2 \gamma_{t-1}^{w-b-1} \text{\qquad by the definition of $\delta_{b-1}$}\\
&= \delta_{b-2}^{n+t-1} \delta_{b-2} \gamma_{n+t-1}^2 \gamma_{t-1}^{w-b-1} \text{\qquad by destabilizing the $\sigma_{b-1}$ term}\\
&= \delta_{b-2}^{n+t} \ \gamma_{n+t-1}^2 \ \gamma_{t-1}^{w-b-1} \\
&= \underline{\delta_{n+t} ^{n+t}} \ \gamma_{n+t-1}^{b-1-(n+t)+1} \ \gamma_{t-1}^{w-b+1}\text{\qquad after repeating this process another $(b-1)-(n+t)-1$ times}\\
&= \underline{\delta_{n+t-1} \delta_{n+t}^{n+t-1} \sigma_{n+t-1}} \ \gamma_{n+t-1}^{b-n-t} \ \gamma_{t-1}^{w-b+1} \text{\qquad by \Cref{prop:SameSuperSub}} \\
&= \delta_{n+t-1} \underline{\delta_{n+t}^{n+t-1}} \sigma_{n+t-1} \ \gamma_{n+t-1}^{b-n-t} \ \gamma_{t-1}^{w-b+1} \\
&= \delta_{n+t-1} \underline{\delta_{n+t-1}^{n+t-2} \delta_{n+t} \gamma_{n+t-2}} \sigma_{n+t-1} \ \gamma_{n+t-1}^{b-n-t} \ \gamma_{t-1}^{w-b+1} \text{\qquad by \Cref{prop9}} \\
&= \delta_{n+t-1} \delta_{n+t-1}^{n+t-2} \underline{\delta_{n+t}} \gamma_{n+t-2} \sigma_{n+t-1} \ \gamma_{n+t-1}^{b-n-t} \ \gamma_{t-1}^{w-b+1} \\
&= \delta_{n+t-1} \delta_{n+t-1}^{n+t-2} \underline{\sigma_{n+t}\delta_{n+t-1}} \gamma_{n+t-2} \sigma_{n+t-1} \ \gamma_{n+t-1}^{b-n-t} \ \gamma_{t-1}^{w-b+1} \\
&= \delta_{n+t-1} \delta_{n+t-1}^{n+t-2} \underline{\delta_{n+t-1}} \gamma_{n+t-2} \sigma_{n+t-1} \ \gamma_{n+t-1}^{b-n-t} \ \gamma_{t-1}^{w-b+1} \text{\qquad by destablizing the $\sigma_{n+t}$ term} \\
&= \delta_{n+t-1}^{n+t} \underline{\gamma_{n+t-2} \sigma_{n+t-1}} \ \gamma_{n+t-1}^{b-n-t} \ \gamma_{t-1}^{w-b+1} \\
&= \delta_{n+t-1}^{n+t} \gamma_{n+t-1} \ \gamma_{n+t-1}^{b-n-t} \ \gamma_{t-1}^{w-b+1} \\
&= \delta_{n+t-1}^{n+t} \gamma_{n+t-1}^{b-n-t+1} \gamma_{t-1}^{w-b+1}.
\end{align*} 

This braid word in $B_{n+t}$ contains a full twist; applying \Cref{thm:FMW}, we deduce the braid index is $n + t$. \end{proof}

\section{Future Directions}
Our proof of \Cref{thm:NBridge} is self-contained and effective: we started with the definition of an $n$-bridge braid, and we produced a Markov equivalent positive braid containing a full twist. The algorithm we produce could be extended to all T-links. 
However, the computations are significantly more tedious, so we do not include them here. An interesting future direction would be to write a computer implementation of our algorithm for all T-links.

\bibliographystyle{amsalpha2}
\bibliography{masterbiblio}

\newcommand{\etalchar}[1]{$^{#1}$}
\providecommand{\bysame}{\leavevmode\hbox to3em{\hrulefill}\thinspace}
\providecommand{\MR}{\relax\ifhmode\unskip\space\fi MR }
\providecommand{\MRhref}[2]{%
  \href{http://www.ams.org/mathscinet-getitem?mr=#1}{#2}
}
\providecommand{\href}[2]{#2}
\begin{thebibliography}{CFK{\etalchar{+}}11}

\bibitem[Ale23]{Alexander}
J.W. Alexander, \emph{A lemma on a system of knotted curves}, Proc. Nat. Acad.
  Sci \textbf{9} (1923), 93--95.

\bibitem[Ber91]{Berge:SolidTorus}
John Berge, \emph{The knots in {$D^2\times S^1$} which have nontrivial {D}ehn
  surgeries that yield {$D^2\times S^1$}}, Topology Appl. \textbf{38} (1991),
  no.~1, 1--19.

\bibitem[Ber18]{Berge}
John Berge, \emph{Some knots with surgeries yielding lens spaces},
  https://arxiv.org/abs/1802.09722 (2018).

\bibitem[BK09]{BirmanKofman}
Joan Birman and Ilya Kofman, \emph{A new twist on {L}orenz links}, J. Topol.
  \textbf{2} (2009), no.~2, 227--248.

\bibitem[Bir13]{Birman:Lorenz}
Joan~S. Birman, \emph{The mathematics of {L}orenz knots}, Topology and dynamics
  of chaos, World Sci. Ser. Nonlinear Sci. Ser. A Monogr. Treatises, vol.~84,
  World Sci. Publ., Hackensack, NJ, 2013, pp.~127--148.

\bibitem[BB05]{BirmanBrendle}
Joan~S. Birman and Tara~E. Brendle, \emph{Braids: a survey}, Handbook of knot
  theory, Elsevier B. V., Amsterdam, 2005, pp.~19--103.

\bibitem[BW83]{BirmanWilliams}
Joan~S. Birman and R.~F. Williams, \emph{Knotted periodic orbits in dynamical
  systems. {I}. {L}orenz's equations}, Topology \textbf{22} (1983), no.~1,
  47--82.

\bibitem[CFK{\etalchar{+}}11]{ChampanerkarFuterKofmanNeumannPurcell}
Abhijit Champanerkar, David Futer, Ilya Kofman, Walter Neumann, and Jessica~S.
  Purcell, \emph{Volume bounds for generalized twisted torus links}, Math. Res.
  Lett. \textbf{18} (2011), no.~6, 1097--1120.

\bibitem[dP22]{DePaiva}
Thiago de~Paiva, \emph{Unexpected essential surfaces among exteriors of twisted
  torus knots}, Algebr. Geom. Topol. \textbf{22} (2022), no.~8, 3965--3982.

\bibitem[dPP22]{dePaivaPurcell}
Thiago de~Paiva and Jessica~S Purcell, \emph{{Satellites and Lorenz knots}},
  International Mathematics Research Notices (2022).

\bibitem[Deh15]{Dehornoy:Lorenz}
Pierre Dehornoy, \emph{On the zeroes of the {A}lexander polynomial of a
  {L}orenz knot}, Ann. Inst. Fourier (Grenoble) \textbf{65} (2015), no.~2,
  509--548.

\bibitem[FS80]{FintushelStern}
Ronald Fintushel and Ronald~J. Stern, \emph{Constructing lens spaces by surgery
  on knots}, Math. Z. \textbf{175} (1980), no.~1, 33--51.

\bibitem[FW87]{FranksWilliams}
John Franks and R.~F. Williams, \emph{Braids and the {J}ones polynomial},
  Trans. Amer. Math. Soc. \textbf{303} (1987), no.~1, 97--108.

\bibitem[Gab89]{Gabai:SolidTori}
David Gabai, \emph{Surgery on knots in solid tori}, Topology \textbf{28}
  (1989), no.~1, 1--6.

\bibitem[Gab90]{Gabai:1BridgeBraids}
\bysame, \emph{{$1$}-bridge braids in solid tori}, Topology Appl. \textbf{37}
  (1990), no.~3, 221--235.

\bibitem[Gar69]{Garside}
F.~A. Garside, \emph{The braid group and other groups}, Quart. J. Math. Oxford
  Ser. (2) \textbf{20} (1969), 235--254.

\bibitem[GM11]{GonzalesMeneses:BraidGroups}
Juan Gonz\'{a}lez-Meneses, \emph{Basic results on braid groups}, Ann. Math.
  Blaise Pascal \textbf{18} (2011), no.~1, 15--59.

\bibitem[GLV18]{GLV:11LSpace}
Joshua~Evan Greene, Sam Lewallen, and Faramarz Vafaee, \emph{{$(1,1)$}
  {L}-space knots}, Compos. Math. \textbf{154} (2018), no.~5, 918--933.

\bibitem[KM22]{KrishnaMorton}
Siddhi Krishna and Hugh Morton, \emph{Twist positivity, lorenz knots, and
  concordance}, Submitted., 2022.

\bibitem[LV19]{LeeVafaee}
Cristine Lee and Faramarz Vafaee, \emph{On 3-braids and {L}-space knots},
  https://arxiv.org/abs/1911.01289 (2019).

\bibitem[Lic62]{Lickorish:DehnSurgery}
W.~B.~R. Lickorish, \emph{A representation of orientable combinatorial
  {$3$}-manifolds}, Ann. of Math. (2) \textbf{76} (1962), 531--540.

\bibitem[Mar35]{Markov}
A.A. Markov, \emph{Uber die freie aquivalenz geschlossener zopfe}, Recueil
  Mathematique Moscou, \textbf{1} (1935), 73--78.

\bibitem[Mor86]{Morton:KnotPolys}
H.~R. Morton, \emph{Seifert circles and knot polynomials}, Math. Proc.
  Cambridge Philos. Soc. \textbf{99} (1986), no.~1, 107--109.

\bibitem[Mor88]{Morton:PolysFromBraids}
\bysame, \emph{Polynomials from braids}, Braids ({S}anta {C}ruz, {CA}, 1986),
  Contemp. Math., vol.~78, Amer. Math. Soc., Providence, RI, 1988,
  pp.~575--585.

\bibitem[Mos71]{Moser:TorusKnots}
Louise Moser, \emph{Elementary surgery along a torus knot}, Pacific J. Math.
  \textbf{38} (1971), 737--745.

\bibitem[Nie20]{Nie:1BBsSatellites}
Zipei Nie, \emph{On $1$-bridge braids, satellite knots, the manifold $v2503$
  and non-left-orderable surgeries and fillings}.

\bibitem[OS05]{OSz:LensSpaceSurgeries}
Peter Ozsv\'ath and Zolt\'an Szab\'o, \emph{On knot {F}loer homology and lens
  space surgeries}, Topology \textbf{44} (2005), no.~6, 1281--1300.

\bibitem[Vaf15]{Vafaee:TwistedTorusKnots}
Faramarz Vafaee, \emph{On the knot {F}loer homology of twisted torus knots},
  Int. Math. Res. Not. IMRN (2015), no.~15, 6516--6537.

\bibitem[Wal60]{Wallace:DehnSurgery}
Andrew~H. Wallace, \emph{Modifications and cobounding manifolds}, Canadian J.
  Math. \textbf{12} (1960), 503--528.

\end{thebibliography}

\vspace{1cm}

\end{document}